\documentclass[11pt]{article}
\usepackage[latin1]{inputenc}
\usepackage[T1]{fontenc}
\usepackage[english]{babel}
\usepackage{amsfonts} 
\usepackage{amsmath}
\usepackage{amsthm}
\usepackage{amssymb}
\usepackage{stmaryrd}
\usepackage{calc}
\usepackage{graphicx}
\usepackage[colorlinks=true,linkcolor=blue,urlcolor=blue,citecolor=blue]{hyperref}
\usepackage{fullpage}

\newtheorem{theo}{Theorem}    

\newtheorem{lem}{Lemma}

\newcommand{\rr}{\mathbb{R}}
\newcommand{\cc}{\mathbb{C}}

\newcommand{\dis}{\displaystyle}
\newcommand{\re}{{\rm{Re}}}
\newcommand{\im}{{\rm{Im}}}
\newcommand{\diff}{{\rm{d}}}
\newcommand{\id}{{\rm{id}}}
\newcommand{\pa}{\partial}

\newcounter{rmq}
\title{A bifurcation analysis for the Lugiato--Lefever equation}
\author{Cyril Godey\\{\small Laboratoire de Math\'ematiques de Besan\c{c}on, Universit\'e de Franche-Comt\'e,}\\{\small 16 route de Gray, 25030 Besan\c{c}on Cedex, France}\\{\small cyril.godey@univ-fcomte.fr}}

\date{}

\begin{document}

\maketitle

\begin{abstract}
The Lugiato--Lefever equation is a cubic nonlinear Schr{\"o}dinger equation, including damping, detuning and driving, which arises as a model in nonlinear optics. We study the existence of stationary waves which are found as solutions of a four-dimensional reversible dynamical system in which the evolutionary variable is the space variable. Relying upon tools from
bifurcation theory and normal forms theory, we discuss the codimension 1 bifurcations. We prove the existence of various types of steady solutions, including spatially localized, periodic, or quasi-periodic solutions.
\smallskip

\end{abstract}

\section{Introduction}

We consider the Lugiato--Lefever equation
\begin{equation}
\frac{\pa \psi}{\pa t}=-i\beta \frac{\pa^2\psi}{\pa x^2}-(1+i\alpha)\psi +i \psi \left|\psi\right|^2+F, \label{lle}
\end{equation}
where the unknown function $\psi$ is complex valued and depends upon the time $t$ and the space variable~$x$. The coefficient $\alpha \in \rr$ is a detuning parameter, $F>0$ is a driving term and $\beta \in \rr^*$ denotes a dispersion parameter, which can be supposed to be equal to $\pm 1$, up to rescaling $x$. In the sequel we refer to the case $\beta=1$ (respectively $\beta=-1$) as the normal (respectively anomalous) dispersion case \cite{godey2014stability}. 

This equation has been derived as a model in several contexts in nonlinear optics \cite{chembo2013spatiotemporal, lugiato1987spatial} and has been intensively studied in the physics literature in the recent years (\textit{e.g.}, see \cite{godey2014stability} and the references therein). In contrast, there are few mathematical results. The constant solutions of the equation~(\ref{lle}) and their temporal stability are well-known (see for example \cite{godey2014stability} and the references therein). In \cite{miyaji2010bifurcation}, using a spatial dynamics approach and tools from bifurcation theory, it has been proved that for $\beta=-1$, the equation (\ref{lle}) possesses several types of stationary solutions, bifurcating from constant solutions. Using a similar formulation, a local bifurcation analysis for stationary waves has been performed in \cite{godey2014stability} in the cases $\beta=1$ and $\beta=-1$. The existence of $2\pi-$periodic stationary solutions satisfying Neumann boundary conditions has been proved in \cite{mandel2016}, provided that the coefficients $\alpha$ and $F$ belong to suitable ranges. Moreover, the stability of a family of periodic solutions has been discussed in \cite{miyaji2011stability}.

In this paper, we present a systematic study of the local bifurcations found in \cite{godey2014stability} and show the existence of various steady solutions of (\ref{lle}), \textit{i.e.}, solutions of the stationary equation
\begin{equation}
\beta \frac{\diff ^2 \psi}{\diff x^2}=(i-\alpha)\psi+\psi |\psi|^2-iF. \label{llestat}
\end{equation} 
Our purpose is to confirm the existence of solutions of the equation (\ref{lle}), which are found numerically and experimentally \cite{godey2014stability}. Our analysis of (\ref{llestat}) relies upon a spatial dynamics approach, which consists in writing the equation (\ref{llestat}) as a dynamical system in which the unbounded space variable $x$ is taken as an evolutionary variable. The existence problem is then analyzed using tools from dynamical systems theory and bifurcation theory. This method has been introduced by K. Kirchg\"assner \cite{MR662490} and has been extensively used to study the existence of traveling waves in many different contexts (\textit{e.g.}, see \cite[Chapter 5]{haragus2010local}). Starting from the formulation of the stationary equation (\ref{llestat}) as a reversible dynamical system (see equation (\ref{syst}) below), the local bifurcations from constant solutions have been classified in \cite{godey2014stability}. Depending on the values of the parameters $\alpha$ and $F$, there are three codimension 1 bifurcations: $(i\omega)^2$, $0^2(i\omega)$ and $0^2$ and one codimension 2 bifurcation, which is an $0^4$ bifurcation. Here we discuss the set of bounded solutions which arise due to these bifurcations. We restrict to the codimension 1 bifurcations, which are well understood \cite[Chapter 4]{haragus2010local}. 
The key step in the analysis of the $(i\omega)^2$ and $0^2(i\omega)$ bifurcations is the computation of their normal form, which consists in finding a simpler form of the Taylor expansion of the nonlinear part of the system.~The computation of the relevant parameters in these expansions, together with the reversibility of the system (\ref{syst}), allows to conclude to the existence of solutions from the general results of \cite[Chapter 4]{haragus2010local}. A center manifold reduction is required to obtain the normal form of the $0^2$ bifurcation. This method allows to reduce the dimension of the system and to find a locally invariant manifold, containing the set of small bounded solutions.

The paper is organized as follows. In Section \ref{s1} we describe the set of constant solutions of the equation (\ref{lle}) and recall some results of the bifurcation analysis in \cite{godey2014stability}. In Sections \ref{s2}, \ref{s3} and \ref{s4} we compute the normal forms of the codimension 1 bifurcations and discuss the existence of several types of steady solutions of the equation (\ref{lle}). We also compute the leading order term in the Taylor expansions of the periodic and homoclinic solutions. We conclude with a discussion of some further issues. 

\noindent\textbf{Acknowledgements. }This work was supported by a doctoral grant of the Franche-Comt\'e region and the LabEx ACTION (project AMELL).

\section{Spatial dynamics and bifurcation diagrams} \label{s1}
In this section we recall the main results of the bifurcation analysis performed in \cite{godey2014stability}.

\vspace{0.3cm}
\noindent\textbf{Constant solutions. }Let $\psi =\psi_r + i \psi_i $ be a constant solution of (\ref{lle}), where $\psi_r$ and $\psi_i $ denote the real and imaginary part of $\psi$, respectively. Then $\psi$ satisfies the algebraic equation
\[
i\psi |\psi|^2-(1+i\alpha)\psi+F=0 .
\]
Notice that $\psi$ does not depend upon $\beta$. A straightforward computation shows that $\psi_r$ and $\psi_i$ are given implicitly by
\[
\psi_r=\frac{F}{1+(\rho-\alpha)^2},\qquad \psi_i=\frac{F(\rho-\alpha)}{1+(\rho-\alpha)^2}, 
\]
where $\rho=|\psi|^2$. In particular, we have
\[
F^2=\rho\left(1+(\rho-\alpha)^2\right).\label{polynome}
\]
This relation relates the parameters $\alpha$, $F^2$ and the square modulus $\rho$ of a constant solution, and allows to determine the number of equilibria of (\ref{lle}) in terms of the parameters $\alpha$ and $F^2$. 

More precisely, for $\alpha \leqslant \sqrt 3$, the cubic polynomial $\rho \longmapsto \rho\left(1+(\rho-\alpha)^2\right)$ is monotonically increasing for $\rho>0$, so that the equation (\ref{polynome}) has precisely one solution, and therefore the equation (\ref{lle}) has one constant solution, for any $F>0$. For $\alpha>\sqrt 3$, the polynomial $\rho \longmapsto \rho\left(1+(\rho-\alpha)^2\right)$ has two positive critical points, which are the roots of the quadratic polynomial $\rho\longmapsto 3\rho^2-4\alpha \rho+\alpha^2+1$,
\[
\rho_{\pm}(\alpha)=\frac{2\alpha\mp\sqrt{\alpha^2-3}}{3}.
\]
(see Figure \ref{monot}). The corresponding values $F^2_{\pm}(\alpha)$ of $F^2$ are then given by
\begin{eqnarray*}
F^2_{\pm}(\alpha)&=&\rho_{\pm}(\alpha)(1+(\rho_{\pm}(\alpha)-\alpha)^2)\\
&=&\frac{2\alpha\mp \sqrt{\alpha^2-3}}{3}\left(1+\left(\frac{\sqrt{\alpha^2-3}\pm\alpha}{3}\right)^2\right).
\end{eqnarray*}

\begin{figure}
\centering
\includegraphics[scale=0.5]{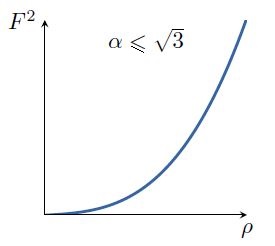}
\includegraphics[scale=0.5]{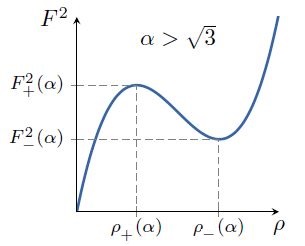}
\caption{Monotonicity of the polynomial $\rho\longmapsto \rho\left(1+(\rho-\alpha)^2\right)$ for $\alpha\leqslant \sqrt 3$ (to the left) and $\alpha > \sqrt 3$ (to the right).}
\label{monot}
\end{figure}
\noindent Consequently, for any $\alpha > \sqrt 3$ and any $F^2\in [F_-^2(\alpha),F_+^2(\alpha)]$, the equation (\ref{polynome}) has three solutions, so that the equation (\ref{lle}) has three constant solutions $\psi_1$, $\psi_2$ and $\psi_3$. Their respective modulus $\rho_1$, $\rho_2$ and $\rho_3$ satisfy
\[
\rho_1<\rho_+(\alpha)<\rho_2<\rho_-(\alpha)<\rho_3.
\]
\begin{figure}
\includegraphics[scale=0.22]{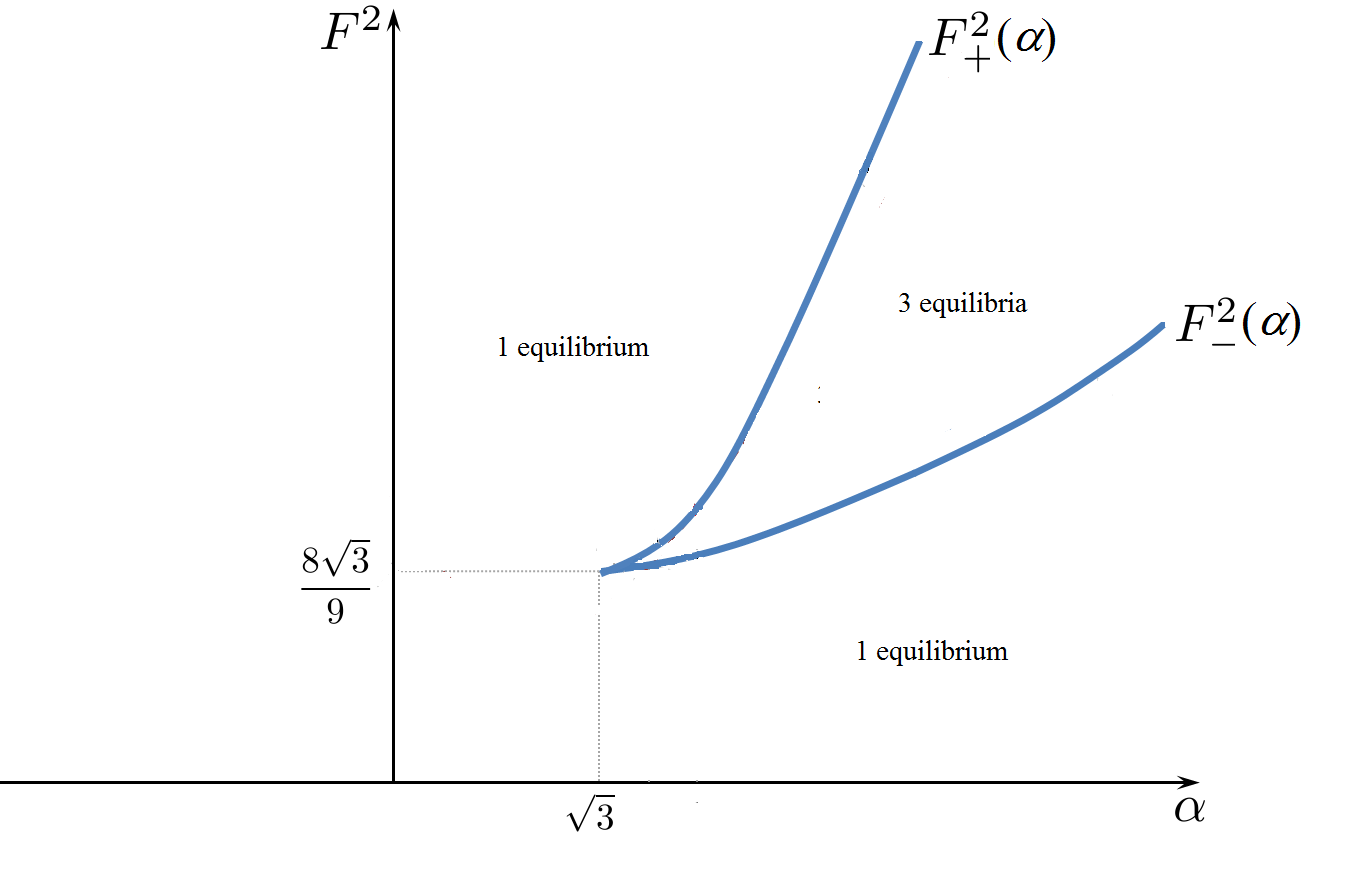}\centering
\caption{Number of equilibria of the equation (\ref{lle}) in terms of $\alpha$ and $F^2$.}
\label{numeq}
\end{figure}
\noindent Moreover, for $F^2\notin [F_-^2(\alpha),F_+^2(\alpha)]$, the equation (\ref{lle}) has one constant solution. We sum up these results in Figure \ref{numeq}.

\vspace{0.3cm}
\noindent\textbf{Dynamical system. }Let $\psi^*=\psi_r ^*+i\psi_i ^*$ be a constant solution of (\ref{lle}), and let $\alpha^*, \ F^*$ be the corresponding values of the parameters. We set
\[
\psi=\psi^*+\widetilde{\psi}, \qquad \widetilde{\psi}=\widetilde{\psi_r}+i\widetilde{\psi_i},\qquad \frac{\diff \widetilde{\psi}}{\diff x}=\widetilde{\varphi_r}+i\widetilde{\varphi_i},
\]
and we rewrite the stationary equation (\ref{llestat}) as a dynamical system in $\rr^4$,
\begin{equation}
\frac{\diff U}{\diff x}=L_{\pm}U+R_{\pm}(U,\alpha,F),\label{syst}
\end{equation}
in which
\[ U=\begin{pmatrix}
\widetilde{\psi_r}\\ \widetilde{\varphi_r} \\ \widetilde{\psi_i}\\ \widetilde{\varphi_i}
\end{pmatrix},\qquad L_{\pm}=\begin{pmatrix}
0 & 1 & 0 & 0 \\
\pm\left(3\psi_r^{*2}+\psi_i^{*2}-\alpha^*\right) & 0 & \pm\left(2\psi_r^* \psi_i ^*-1 \right)& 0 \\
0 & 0 & 0 & 1 \\
\pm\left(2\psi_r^*\psi_i^*+1\right) & 0 & \pm\left(\psi_r^{*2}+3\psi_i^{*2}-\alpha^*\right) & 0
\end{pmatrix},
\]
and
\[
R_{\pm}(U,\alpha,F)=\pm \begin{pmatrix}
0\\ \widetilde{\psi_r}^3+\widetilde{\psi_r}\widetilde{\psi_i}^2+3\psi_r^* \widetilde{\psi_r}^2+2\psi_i^* \widetilde{\psi_r}\widetilde{\psi_i}+\psi_r^*\widetilde{\psi_i}^2-(\alpha-\alpha^*)(\widetilde{\psi_r}+\psi_r^*) \\ 0 \\
\widetilde{\psi_i}^3+\widetilde{\psi_r}^2 \widetilde{\psi_i}+\psi_i^* \widetilde{\psi_r}^2+2\psi_r^* \widetilde{\psi_r}\widetilde{\psi_i}+3\psi_i^*\widetilde{\psi_i}^2-(\alpha-\alpha^*)(\widetilde{\psi_i}+\psi_i^*)-(F-F^*)
\end{pmatrix}.
\]
Here the symbols $+$ and $-$ in $L_{\pm}$ and $R_{\pm}$ stand for the case of normal ($\beta=1$) and anomalous ($\beta=-1$) dispersion, respectively. 

Remark that the system (\ref{syst}) is reversible, \textit{i.e.}, there exists a symmetry $S=\rm{diag }(1,-1,1,-1)$ which anticommutes with the vector field in (\ref{syst}):
\[
L_{\pm}SU=-SL_{\pm}U,\quad R_{\pm}(SU,\alpha,F)=-SR_{\pm}(U,\alpha,F),\quad \forall \  U\in \rr^4,\ \forall \ (\alpha,F) \in \rr\times (0,+\infty).
\]
\textbf{Local bifurcations. }Local bifurcations are determined by the purely imaginary eigenvalues of the matrix $L_{\pm}$. We denote by $\sigma_{\rm{im}}(L_\pm)$ the set of purely imaginary eigenvalues of $L_\pm$. A direct calculation shows that the eigenvalues of $L_{\pm}$ are the complex roots $X$ of the characteristic polynomial
\[
X^4\mp(4\rho^*-2\alpha^{*2})X^2+3\rho^{*2}-4\alpha^* \rho^*+\alpha^{*2}+1,\quad \rho^*=\psi_r^{*2}+\psi_i^{*2}.
\]
Notice that 0 is a root of this polynomial if and only if
\[
3\rho^{*2}-4\alpha^* \rho^*+\alpha^{*2}+1=0,
\]
\textit{i.e.}, if and only if
$\rho^*=\rho_{\pm}(\alpha^*)$, when $F^{*2}=F_{\pm}^2(\alpha^*)$. In this case $0$ is at least a double eigenvalue of $L_{\pm}$ and the other two eigenvalues of $L_{\pm}$ satisfy
$$X^2=\pm (4\rho^*-2\alpha^*).$$

Restricting to purely imaginary eigenvalues of $L_{\pm}$, it was shown in \cite{godey2014stability} that these exist for values of the parameters $\alpha^*$ and $F^*$ along the curves given in Figure \ref{diagram}. We have four types of bifurcations: 
\begin{itemize}
\item an $(i\omega)^2$ bifurcation when $\sigma_{\rm{im}}(L_{\pm})=\left\lbrace \pm i \omega \right\rbrace$, where $\pm i\omega$ are algebraically double and geometrically simple eigenvalues. In the case of normal dispersion ($\beta=1$), this bifurcation occurs for $\alpha^*>2$ and $F^{*2}=1+(1-\alpha^*)^2$, when $\rho^*=1$. In the case of anomalous dispersion ($\beta=-1$), it occurs for $\alpha^*<2$, $F^{*2}=1+(1-\alpha^*)^2$ and $\rho^*=1$.
\item a $0^2(i\omega)$ bifurcation when $\sigma_{\rm{im}}(L_{\pm})=\left\lbrace0,\pm i\omega \right\rbrace$, where 0 is an algebraically double and geometrically simple eigenvalue. In the case of normal dispersion, this bifurcation occurs for $\alpha ^*>2$, and $F^{*2}=F^2_+(\alpha^*)$, when $\rho^*=\rho_+(\alpha^*)$. In the case of anomalous dispersion, it occurs for $\alpha^* \in (\sqrt 3,2)$ $F^{*2}=F^2_+(\alpha^*)$, and $\rho^*=\rho_+(\alpha^*)$, and for $\alpha^* >\sqrt 3$, $F^{*2}=F^2_-(\alpha^*)$ and $\rho^*=\rho_-(\alpha^*)$.
\item a $0^2$ bifurcation when $\sigma_{\rm{im}}(L_{\pm})=\left\lbrace 0 \right\rbrace$, where 0 is an algebraically double and geometrically simple eigenvalue. In the case of normal dispersion, it occurs for $\alpha^* \in (\sqrt 3,2)$ and $F^{*2}=F^2_+(\alpha^*)$, when $\rho^*=\rho_+(\alpha^*)$, and for $\alpha^* >\sqrt 3$, $F^{*2}=F^2_-(\alpha^*)$, when $\rho^*=\rho_-(\alpha^*)$.  In the case of anomalous dispersion, this bifurcation occurs for $\alpha^* >2$, $F^{*2}=F^2_+(\alpha^*)$ and $\rho^*=\rho_+(\alpha^*)$. 
\item a $0^4$ bifurcation when $\sigma_{\rm{im}}(L_{\pm})=\left\lbrace 0 \right\rbrace$, where 0 is an algebraically quadruple and geometrically simple eigenvalue. In both cases, this bifurcation occurs for $\alpha^*=F^{*2}=2$, when $\rho^*=1$.
\end{itemize}

\noindent\textbf{Remark 1. }The $0^2(i\omega)$ and $0^2$ bifurcations are also found for $\alpha^*=\sqrt 3$, but at this point one of the leading order coefficient of the normal form system vanishes (see Sections \ref{o2iomega2} and \ref{o21}). The analysis of these bifurcations at $\alpha^*=\sqrt 3$ would require a further expansion of the vector field.

\begin{figure}[!h]
\includegraphics[scale=0.23]{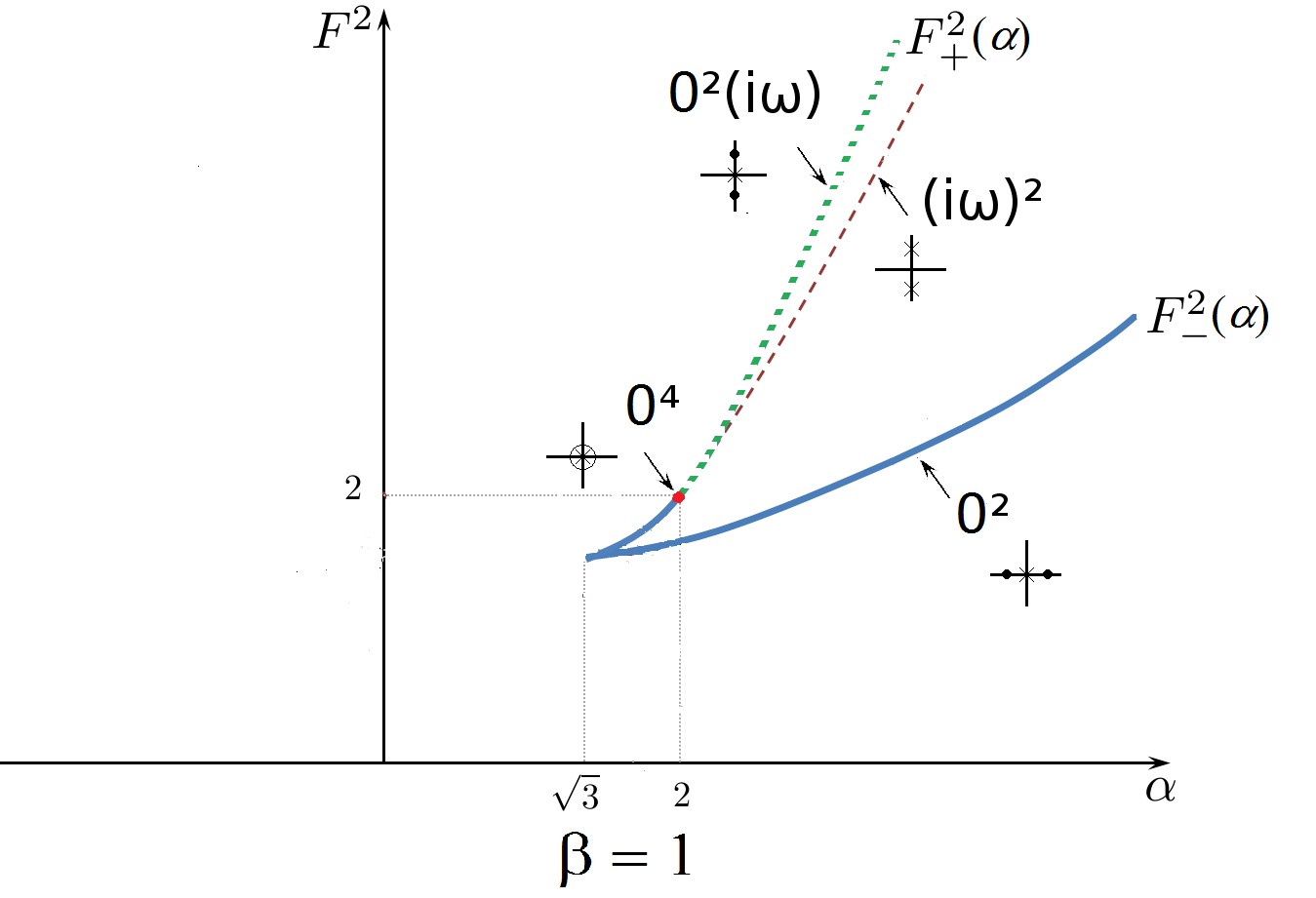}
\includegraphics[scale=0.23]{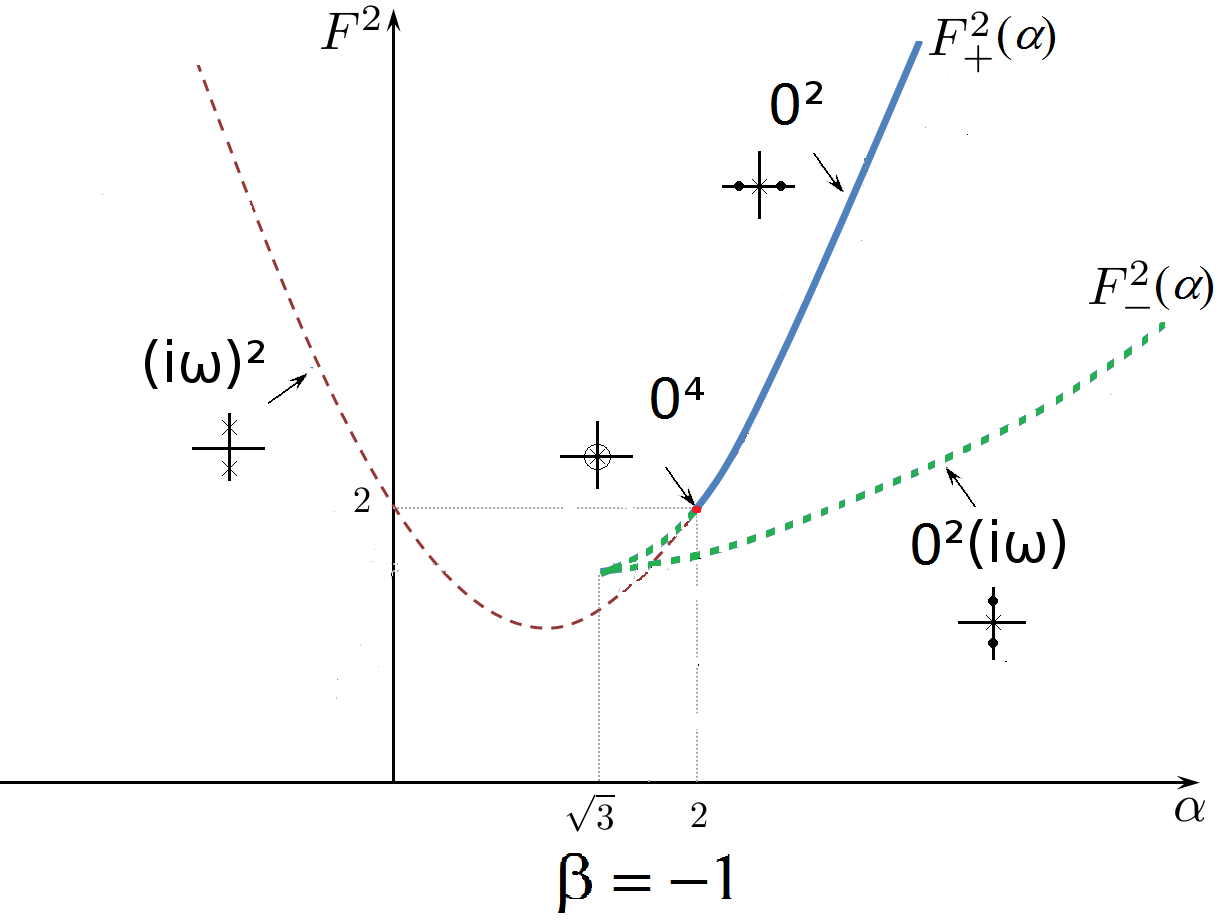}
\caption{Bifurcation diagrams for the case of normal (left) and anomalous (right) dispersion. Simple, double and quadruple eigenvalues are indicated with a dot, a cross and a circled cross, respectively. $(i\omega)^2$ bifurcations occur along the dashed lines, $0^2(i\omega)$ bifurcations along the dotted lines, and $0^2$ bifurcations along the continuous lines. In both cases a $0^4$ bifurcation occurs at the point $(\alpha^*, F^{*2})=(2,2)$.}
\label{diagram}
\end{figure}

The dynamics of the system (\ref{syst}) depends on the type of bifurcation. In the sequel we restrict to the codimension 1 bifurcations: $(i\omega)^2$, $0^2(i\omega)$ and $0^2$. We fix the parameter $F=F^*$ and take $\alpha$ as a bifurcation parameter by setting $\alpha=\alpha^*+\mu$, with small $\mu$. Then the system (\ref{syst}) becomes
\begin{equation}
\frac{\diff U}{\diff x}=L_{\pm}U+R_{\pm}(U,\mu),\label{syst2}
\end{equation}
where $U$, $L_{\pm}$ are the same as in (\ref{syst}) and
\[
R_{\pm}(U,\mu)=\pm \begin{pmatrix}
0\\ \widetilde{\psi_r}^3+\widetilde{\psi_r}\widetilde{\psi_i}^2+3\psi_r^* \widetilde{\psi_r}^2+2\psi_i^* \widetilde{\psi_r}\widetilde{\psi_i}+\psi_r^*\widetilde{\psi_i}^2-\mu(\widetilde{\psi_r}+\psi_r^*) \\ 0 \\
\widetilde{\psi_i}^3+\widetilde{\psi_r}^2 \widetilde{\psi_i}+\psi_i^* \widetilde{\psi_r}^2+2\psi_r^* \widetilde{\psi_r}\widetilde{\psi_i}+3\psi_i^*\widetilde{\psi_i}^2-\mu(\widetilde{\psi_i}+\psi_i^*)
\end{pmatrix}.
\]
The map $R_{\pm}$ is smooth, $R_{\pm}\in \mathcal C ^{\infty}(\mathbb R ^4 \times \mathbb R,\mathbb R^4)$, and 
\[R_{\pm}(0,0)=0, \qquad\diff _U R_{\pm}(0,0)=0,\]
where $\diff _U$ denotes its differential with respect to $U$. Notice that the Taylor expansion of $R_{\pm}$ is given by
\[
R_{\pm}(U,\mu)=\mu R_{\pm}^{0,1}+\mu  R_{\pm}^{1,1}(U)+ R_{\pm}^{2,0}(U,U)+ R_{\pm}^{3,0}(U,U,U),
\]
where
\[
R_{\pm}^{0,1}=\pm\begin{pmatrix}
0 \\ -\psi_r^* \\ 0 \\ -\psi_i^*
\end{pmatrix},\qquad R_{\pm}^{1,1}(U)=\pm\begin{pmatrix}
0 \\ -\widetilde{\psi_r} \\ 0 \\ -\widetilde{\psi_i}
\end{pmatrix},
\]
\[
R_{\pm}^{2,0}(U,U)=\pm\begin{pmatrix}
0 \\ 3\psi_r^* \widetilde{\psi_r}^2+2\psi_i^* \widetilde{\psi_r}\widetilde{\psi_i}+\psi_r^*\widetilde{\psi_i}^2 \\ 0 \\ \psi_i^* \widetilde{\psi_r}^2+2\psi_r^* \widetilde{\psi_r}\widetilde{\psi_i}+3\psi_i^*\widetilde{\psi_i}^2
\end{pmatrix},\qquad R_{\pm}^{3,0}(U,U,U)=\pm\begin{pmatrix}0\\ \widetilde{\psi_r}^3+\widetilde{\psi_r}\widetilde{\psi_i}^2 \\0 \\ \widetilde{\psi_i}^3+\widetilde{\psi_r}^2 \widetilde{\psi_i} \end{pmatrix}.
\]

\section{Reversible $(i\omega)^2$ bifurcation }\label{s2}

In this section we treat the case of the reversible $(i\omega)^2$ bifurcation and compute the normal form in the cases of normal ($\beta=1$) and anomalous ($\beta=-1$) dispersion. We then discuss the existence of bounded solutions of (\ref{syst2}), in particular of periodic and localized solutions.
\subsection{Case of normal dispersion ($\beta=1$)}\label{s21}

Consider the $(i\omega)^2$ bifurcation, which for $\beta=1$ occurs along the half line 
\[\alpha^*>2, \quad F^{*2}=1+(1-\alpha^*)^2\]
in the $(\alpha, F^2)$--parameter space, along which $\rho^*=1$. Recall that $\alpha=\alpha^*+\mu$ is the bifurcation parameter. For such values of the parameters, the matrix $L_+$ possesses two pairs of algebraically double and geometrically simple purely imaginary eigenvalues denoted by $\pm i\omega^*$, with $\omega^*=\sqrt{\alpha^*-2}$. 

Following \cite[Chapter 4, Lemma 3.15]{haragus2010local}, we consider a basis $\left\lbrace\re\ \zeta_0, \im\ \zeta_0, \re \ \zeta_1,\im\ \zeta_1\right\rbrace$ of $\rr^4$ such that
$$(L_+-i\omega^*)\zeta_0=0, \ (L_+-i\omega^*)\zeta_1=\zeta_0, \ (L_+ +i\omega^*)\overline{\zeta_0}=0, \ (L_+ +i\omega^*)\overline{\zeta_1}=\overline{\zeta_0},$$
$$S\zeta_0=\overline{\zeta_0}, \ S\zeta_1=-\overline{\zeta_1}.$$
A direct computation gives
\[
\zeta_0=\begin{pmatrix}
C \\ i\omega^* C \\ 1\\ i\omega^* 
\end{pmatrix}
, \quad
\zeta_1=\begin{pmatrix}
\frac{2i\omega^*}{1+2\psi_r^*\psi_i^*}\\C-\frac{2\omega^{*2}}{1+2\psi_r^*\psi_i^*}\\0\\ 1
\end{pmatrix},
\]
where, using the formulas for $\psi_r^*$ and $\psi_i^*$ and the fact that $\psi_r^{*2}+\psi_i^{*2}=1$, we obtain \[ C=\frac{1-2\psi_r^*\psi_i^*}{3\psi_r^{*2}+\psi_i^{*2}-2}=\frac{\psi_r^*-\psi_i^*}{\psi_r^*+\psi_i^*}=\frac{\alpha^*}{2-\alpha^*}.\] 
In this basis we represent a vector $u\in \rr^4$ by $$\dis u=A\zeta_0 + B\zeta_1 +\overline A \overline{\zeta _0 } +\overline B \overline{\zeta_1}, $$ with $A, B \in \cc$, by identifying $\rr^4$ with the complex vector space $\widetilde{\rr^4}=\left\lbrace  (A,B,\overline A, \overline B), A, B \in \cc \right\rbrace.$

\vspace{0.3cm}
\noindent\textbf{Normal form.} 
The vector field in the reversible system (\ref{syst2}) is of class $\mathcal C^{\infty}$ on $ \rr^4 \times \rr$ and satisfies
\[
R_+(0,0)=0,\quad {\mathrm{d}}_U R_+(0,0)=0.
\]
We apply to the system (\ref{syst2}) the normal form theorem for reversible systems in the case of a $(i\omega)^2$ bifurcation (see \cite[Chapter 3, Theorem 3.4]{haragus2010local} and \cite[Chapter 4, Lemma 3.17]{haragus2010local}). For any positive integer $N \geqslant 2$, there exist two neighborhoods $V_1$ and $V_2$ of 0 in $\widetilde{\rr^4}$ and $\rr$, respectively, and for any $\mu \in V_2$ there exists a polynomial $\Phi(\cdot,\mu):\widetilde{\rr^4} \longrightarrow \widetilde{\rr^4}$ of degree $N$ with the following properties:
\begin{enumerate}
\item the coefficients of the monomials of degree $q$ in $\Phi(\cdot,\mu)$ are functions of $\mu$ of class $\mathcal C^{\infty}$,
\[
\Phi(0,0,0,0,0)=0, \qquad \partial_{(A,B,\overline A,\overline B)}\Phi(0,0,0,0,0)=0,\]
and
\[\Phi(\overline A, -\overline B, A,-B,\mu)=S\Phi(A,B,\overline A,\overline B,\mu);
\]
\item for $(A,B,\overline A,\overline B)\in V_1$ the change of variable 
\begin{equation} 
u=A\zeta_0+B\zeta_1+\overline{A}\overline{\zeta_0}+\overline{B}\overline{\zeta_1}+\Phi(A,B,\overline A, \overline B, \mu)\label{chvar}
\end{equation}
transforms the system (\ref{syst2}) into the normal form
\begin{eqnarray}\label{normal}
\frac{\diff A}{\diff x}&=&i\omega^* A +B +iAP\left(|A|^2,\frac i 2 (A\overline{B}-\overline{A}B), \mu \right)+\rho_A(A,B,\overline A,\overline B,\mu)\nonumber\\
\frac{\diff B}{\diff x}&=&i\omega^* B +iBP\left(|A|^2,\frac i 2 (A\overline{B}-\overline{A}B),\mu\right)+AQ\left(|A|^2,\frac i 2 (A\overline{B}-\overline{A}B),\mu\right)\\ & &+\rho_B(A,B,\overline A,\overline B,\mu),\nonumber
\end{eqnarray}
where $P$ and $Q$ are real-valued polynomials of degree $N-1$ in $(A, B, \overline{A},\overline{B})$. Moreover the remainders $\rho_A$ and $\rho_B$ satisfy
\begin{eqnarray*}
\rho_A(\overline A,-\overline B, A,-B,\mu)&=&-\overline{\rho_A}(A,B,\overline A,\overline B,\mu),\\
\rho_B(\overline A,-\overline B, A,- B,\mu)&=&\overline{\rho_B}(A,B,\overline A,\overline B,\mu)
\end{eqnarray*}
and
$$|\rho_A(A,B,\overline A,\overline B,\mu)|+|\rho_B(A,B,\overline A,\overline B,\mu)|=o\left((|A|+|B|)^N\right).$$
\end{enumerate}

Consider the expansions of $P$ and $Q$:
\begin{eqnarray}\label{expq}
P\left(|A|^2,\frac i 2 (A\overline{B}-\overline{A}B), \mu \right)=a_1\mu+b_1|A|^2+\frac{ic_1}{2}(A\overline{B}-\overline{A}B)+O((|\mu|+(|A|+|B|)^2)^2),\nonumber \\
Q\left(|A|^2,\frac i 2 (A\overline{B}-\overline{A}B), \mu \right)=a_2\mu+b_2|A|^2+\frac{ic_2}{2}(A\overline{B}-\overline{A}B)+O((|\mu|+(|A|+|B|)^2)^2).
\end{eqnarray}
%
According to \cite[Chapter 4]{haragus2010local}, the dynamics of the system (\ref{normal}) is determined by the signs of the coefficients $a_2$ and $b_2$. 
\begin{lem}\label{lemma1}
The coefficients $a_2$ and $b_2$ in the expansion (\ref{expq}) of the polynomial $Q$ are given by
$$a_2=\frac{\alpha^*-1}{(\alpha^*-2)^3}>0, \quad   b_2=\frac{2F^{*2}(41-30\alpha^*)}{9(\alpha^*-2)^5}<0.$$
\end{lem}

\begin{proof}
We follow the method in \cite[Section 4.3.3]{haragus2010local}.

\vspace{0.3cm}
\noindent\textbf{Computation of $a_2$. }
Consider a vector $\zeta_1^*$ such that
$$\left\langle\zeta_0,\zeta_1^*\right\rangle=0, \quad \left\langle\overline{\zeta_0},\zeta_1^*\right\rangle=0 ,\quad \left\langle\zeta_1,\zeta_1^*\right\rangle=1,\quad \left\langle\overline{\zeta_1},\zeta_1^*\right\rangle=0, \quad S \zeta_1^*=-\overline{\zeta_1^*},\quad L_+^*\zeta_1^*=-i\omega^* \zeta_1^*,$$
where $L_+^*$ is the adjoint operator of $L_+$ and $\langle\cdot, \cdot\rangle$ denotes the hermitian scalar product on $\widetilde{\mathbb R^4}$.
A direct computation gives \[ \zeta_1^*=\frac{2-\alpha^*}{4F^{*2}} \begin{pmatrix}
-i\omega^*\\ 1 \\ i\omega^* C \\-C
\end{pmatrix}.\]
According to \cite[Appendix D.2.]{haragus2010local}, the coefficient $a_2$ is given by
\begin{equation}
a_2=\langle R_+^{1,1}\zeta_0+2R_+^{2,0}(\zeta_0,\Phi_{00001}),\zeta_1^*\rangle,\label{a2}
\end{equation}
where $\Phi_{00001}$ is the coefficient of the monomial $\mu$ in the expansion of the polynomial $\Phi$ in (\ref{chvar}), solving
\[
L_+\Phi_{00001}=-R_+^{0,1}.
\]
Using the expression of $R_+^{0,1}$ given in Section 2, we find
\[
\Phi_{00001}=\frac{1}{(\alpha^*-2)^2}\begin{pmatrix}
(1-\alpha^*)\psi_r^*+\psi_i^* \\ 0 \\ -\psi_r^*+(1-\alpha^*)\psi_i^* \\ 0
\end{pmatrix}.
\]
Next, we compute
\[
R_+^{1,1}\zeta_0=\begin{pmatrix}
0 \\ -C \\ 0 \\ -1
\end{pmatrix}
\]
and
\[
R_+^{2,0}(\zeta_0,\Phi_{00001})=\frac{1}{(\alpha^*-2)^2}\begin{pmatrix}
0\\ \left( 3C(1-\alpha^*)-1\right)\psi_r^{*2}+2\left( C+1-\alpha^*\right)\psi_r^{*}\psi_i^*+\left( C(1-\alpha^*)+1\right)\psi_i^{*2}\\ 0 \\ (1-\alpha^*-C)\psi_r^{*2}+2\left( C(1-\alpha^*)-1\right)\psi_r^* \psi_i^*+\left(C+3(1-\alpha^*) \right)\psi_i^{*2}
\end{pmatrix}.
\]
We obtain from (\ref{a2}) 
\[
a_2=\frac{(2C(1-\alpha^*)+C^2-1)(\psi_r^{*2}-\psi_i^{*2})+2(2C+(1-\alpha^*)(1-C^2))\psi_r^* \psi_i^*}{2 F^{*2} (2-\alpha^*)}.
\]
Finally, using the symbolic package Maple, we have
\[
a_2=\frac{\alpha^*-1}{(\alpha^*-2)^3}>0.
\]

\noindent\textbf{Computation of $b_2$.}
According to \cite[Appendix D.2.]{haragus2010local}, the coefficient $b_2$ is given by
\begin{equation}\label{eq2}
b_2=\left\langle 2R_+^{2,0}(\zeta_0,\Phi_{10100})+2R_+^{2,0}(\overline{\zeta_0},\Phi_{20000})+3R_+^{3,0}(\zeta_0,\zeta_0,\overline{\zeta_0}),\zeta_1^* \right\rangle,
\end{equation} 
where $\Phi_{20000}$ and $\Phi_{10100}$ are the coefficients of the monomials $A^2$ and $A\overline A$, respectively, in the expansion of the polynomial $\Phi$ in (\ref{chvar}), solving
\[
(L_+-2i\omega^*)\Phi_{20000}+R_+^{2,0}(\zeta_0,\zeta_0)=0,
\]
\[L_+\Phi_{10100}+2R_+^{2,0}(\zeta_0,\overline{\zeta_0})=0.
\]
Using the formulas for $R_+^{2,0}$ in Section 2, we find
\[
R_+^{2,0}(\zeta_0,\zeta_0)=R_+^{2,0}(\zeta_0,\overline{\zeta_0})=\begin{pmatrix}
0 \\ (3C^2+1)\psi_r^*+2C\psi_i^* \\ 0 \\ 2C \psi_r^* +(C^2+3)\psi_i^*
\end{pmatrix},
\]

and then\[ \Phi_{20000}=\frac{1}{9(\alpha^*-2)^2}\begin{pmatrix}
D\\ 2i\omega^* D \\ E \\ 2i\omega^* E
\end{pmatrix}, \quad  \Phi_{10100}=\frac{1}{(\alpha^*-2)^2}\begin{pmatrix}
G\\0\\H\\0
\end{pmatrix},
\]
where the quantities $D$, $E$, $G$, $H$ are given by
\[
D=-(2\psi_i^{*2}+3\alpha^*-7)((3C^2+1)\psi_r^*+2C\psi_i^*)+(2\psi_r^* \psi_i^*-1)(2C\psi_r^*+(C^2+3)\psi_i^*),
\]
\[
E=-(2\psi_r^{*2}+3\alpha^*-7)(2C\psi_r^*+(C^2+3)\psi_i^*)+(2\psi_r^*\psi_i^*+1)((3C^2+1)\psi_r^*+2C\psi_i^*)
\]
\[
G=-2(2\psi_i^{*2}+1-\alpha^*)((3C^2+1)\psi_r^*+2C\psi_i^*)-2((3C^2+1)\psi_i^*+2C\psi_r^*)(\psi_r^*-\psi_i^*)^2,
\]
and
\[
H=-2(2\psi_r^{*2}+1-\alpha^*)(2C\psi_r^*+(C^2+3)\psi_i^*)+2((3C^2+1)\psi_r^*+2C\psi_i^*)(2\psi_r^*\psi_i^*+1).
\]
Next, we compute
\[
R_+^{2,0}(\zeta_0,\Phi_{10100})=\begin{pmatrix}
0 \\ (3CG+H)\psi_r^*+(CH+G)\psi_i^* \\ 0 \\ (CH+G)\psi_r^*+(CG+3H)\psi_i^*
\end{pmatrix},\]
\[R_+^{2,0}(\overline{\zeta_0},\Phi_{20000})=
\begin{pmatrix}
0 \\ (3CD+E)\psi_r^*+(CE+D)\psi_i^* \\ 0 \\ (CE+D)\psi_r^*+(CD+3E)\psi_i^*
\end{pmatrix},
\quad
R_+^{2,0}(\zeta_0,\Phi_{10100})=\begin{pmatrix}
0\\C^3+C\\0\\C^2+1
\end{pmatrix}.
\]
Finally, using the symbolic package Maple, we obtain from (\ref{eq2})
\begin{eqnarray*}
b_2&=&\frac{2-\alpha^*}{2F^{*2}}\left(\left(2C(D+G)+(1-C^2)(E+H)\right)\psi_r^*+\left((1-C^2)(D+G)-2C(E+H)\right)\psi_i^*\right)\\
&=&\frac{2F^{*2}(41-30\alpha^*)}{9(\alpha^*-2)^5}<0.
\end{eqnarray*}
This completes the proof of the lemma.
\end{proof}

The reversibility of the system (\ref{syst2}) is a key argument in the proofs of the existence of different solutions (see \cite[Section III]{iooss1993perturbed} for more details). According to \cite[Chapter 4, Theorem 3.21]{haragus2010local}, we obtain the following result.

\begin{theo}\label{thmiomega2}Suppose that $\alpha^*>2$ and $F^{*2}=1+(1-\alpha^*)^2$. 
\begin{enumerate}
\item For any $\mu>0$ sufficiently small, the system (\ref{syst2}) possesses one symmetric equilibrium, a one-parameter family of periodic orbits, a two-parameter family of quasiperiodic orbits and a pair of reversible homoclinic orbits to the symmetric equilibrium. 
\item For any $\mu <0$ sufficiently small, the system (\ref{syst2}) possesses one symmetric equilibrium, a one-parameter family of periodic orbits and a two-parameter family of quasiperiodic orbits. 
\end{enumerate}
\end{theo}

\noindent\textbf{Solutions. }Going back to the equation (\ref{lle}), we can compute the leading order terms in the Taylor expansions of these solutions. We focus on the periodic and homoclinic solutions. 
First, the origin is an equilibrium of the normal form (\ref{normal}), which corresponds to a symmetric equilibrium of the system (\ref{syst2}), and to a constant solution of (\ref{lle}) with corresponding parameters $\alpha^*+\mu$ and $F^*$, given by
\[
\psi=\psi^*+O(\mu).
\] 
The leading order terms of the non-constant solutions are computed from the truncated system
\begin{eqnarray}
\frac{\diff A}{\diff x}&=&i\omega^* A +B \label{trunciomega}\\
\frac{\diff B}{\diff x}&=&i\omega^* B +a_2\mu A+b_2 |A|^2 A.\label{trunciomega1}
\end{eqnarray}

In the cases $\mu>0$ and $\mu<0$, the system (\ref{trunciomega})-(\ref{trunciomega1}) has a family of periodic solutions of the form
\[
A(x)=\sqrt{\frac{-a_2 \mu-K^2}{b_2}}e^{i(\omega^* +K)x},\quad B(x)=iK\sqrt{\frac{-a_2 \mu-K^2}{b_2}}e^{i(\omega^* +K)x},
\]
which are parametrized by some $K\in \mathbb R$, such that $a_2 \mu -K^2>0$.
The corresponding solutions of the equation (\ref{lle}) have the form
\[
\psi_{\mu,K}(x)=P(kx),
\]
where the $2 \pi-$periodic function $P$ is given by
\[
P(y)=\psi^*+2\sqrt{\frac{-a_2 \mu-K^2}{b_2}}\left(C-\frac{4\omega^*}{1+2\psi_r^* \psi_i ^*}+i\right)\cos( y)+O(\mu),
\]
and $k=\omega^*+K+O(\mu^{\frac 1 2}).$

In the case $\mu>0$, the reversible homoclinic solutions of the system (\ref{trunciomega})-(\ref{trunciomega1}) have the form
\[
A(x)=\sqrt{\frac{-2a_2 \mu}{b_2}}\,{\rm{sech}} (\sqrt{a_2 \mu}x)e^{i\omega^* x},\quad B(x)=-\sqrt{\frac{-2}{b_2}}a_2 \mu \tanh(\sqrt{a_2 \mu}x)\,{\rm{sech}} (\sqrt{a_2 \mu}x)e^{i\omega^* x}.
\]
The corresponding solutions of (\ref{lle}) read
\begin{multline*}
\psi_{\mu}(x)=\psi^*+2(C+i)\sqrt{\frac{-2a_2 }{b_2}}\,{\rm{sech}} (\sqrt{a_2 \mu}x)\cos(\omega^* x)\mu^{\frac 1 2}\\+\frac{4 \omega^*}{1+\psi_r^* \psi_i ^*}\sqrt{\frac{-2}{b_2}}a_2  \tanh(\sqrt{a_2 \mu}x)\,{\rm{sech}} (\sqrt{a_2 \mu}x)\sin(\omega^* x)\mu+O(\mu^{\frac 3 2}).
\end{multline*}

\subsection{Case of anomalous dispersion ($\beta=-1$)}

We consider the case $\beta=-1$ and keep the same notations as in the previous section. We suppose that $\alpha^*<2$, $F^{*2}=1+(1-\alpha^*)^2$ and $\rho^*=1$. The $(i\omega)^2$ bifurcation which occurs in this case has been studied in \cite{miyaji2010bifurcation}, using a different formulation of the dynamical system. With this latter formulation it is more convenient to take $\rho$ as a bifurcation parameter. We briefly discuss below this bifurcation in our setting, with $\alpha$ as a bifurcation parameter. 

As in the previous section we first compute a suitable basis $\left\lbrace\re\ \zeta_0, \im\ \zeta_0, \re \ \zeta_1,\im\ \zeta_1\right\rbrace$ of $\rr^4$, where now

$$\zeta_0=\begin{pmatrix}
C\\i\omega^* C \\ 1 \\ i\omega^*
\end{pmatrix},\qquad \zeta_1=\begin{pmatrix}
-\frac{2i\omega^*}{1+2\psi_r^* \psi_i ^*}\\ C+ \frac{2\omega^{*2}}{1+2\psi_r^*\psi_i^*}\\0\\1
\end{pmatrix},\qquad \omega^*=\sqrt{2-\alpha^*},\qquad C=\frac{\alpha^*}{2-\alpha^*}.$$
 
Following the same arguments, the system (\ref{syst2}) has the normal form (\ref{normal}). The coefficients $a_2$ and $b_2$ are computed as in the proof of Lemma \ref{lemma1}.

\begin{lem} The coefficients $a_2$ and $b_2$ in the expansion (\ref{expq}) of the polynomial $Q$ are given by
  $$ a_2=\frac{\alpha^*-1}{(2-\alpha^*)^3}, \qquad b_2=\frac{2F^{*2}(41-30\alpha^*)}{9(2-\alpha^*)^5}.$$
\end{lem}

Notice that the signs of these coefficients depend upon $\alpha^*$. They change at $\alpha^*=1$ and $\alpha^*=41/30$. Then, from \cite[Chapter 4, Theorem 3.21]{haragus2010local} we obtain the following result.
\begin{theo}
\begin{enumerate}
\item Suppose that $ 41/30<\alpha^*<2$, $\rho^*=1$ and $F^{*2}=1+(1-\alpha^*)^2$. Then $a_2>0, b_2<0$, and for any $\mu$ sufficiently small, the system (\ref{syst2}) possesses the same types of solutions as in Theorem \ref{thmiomega2}.
\item Suppose that $ \alpha^* < 41/30,\  \alpha^* \neq 1$, $\rho^*=1$ and $F^{*2}=1+(1-\alpha^*)^2$. Then $b_2>0$.
\begin{enumerate}
\item For any $\mu$ sufficiently small such that $a_2 \mu >0$, there is one symmetric equilibrium and no other bounded solution.
\item For any $\mu$ sufficiently small such that $a_2 \mu <0$, the system (\ref{syst2}) has one symmetric equilibrium, a one-parameter family of periodic orbits and a two-parameter family of quasiperiodic orbits. Moreover, there exists a one-parameter family of homoclinic orbits to periodic orbits. 
\end{enumerate}
\end{enumerate}
\end{theo}

\noindent\textbf{Remark 2. }The case $\alpha^*=1$ does not enter in the analysis, since the coefficient $a_2$ vanishes. This issue can be avoided by fixing $\alpha=\alpha^*$, and taking $F^2=F^{*2}+\mu$ as a bifurcation parameter.

\vspace{0.3cm}
\noindent\textbf{Solutions. }As before we can compute the leading order terms in the expansions of the corresponding solutions of the equation (\ref{lle}). We obtain similar formulas for solutions of the truncated system (\ref{trunciomega})-(\ref{trunciomega1}). We give hereafter the expansions of the corresponding solutions of the equation (\ref{lle}).

First, in the case $ 41/30<\alpha^*<2$, the solutions have the same expansions as the one found in Section \ref{s21}. 

Next, in the case $ \alpha^* < 41/30$ and $a_2 \mu<0$, the family of periodic solutions of the system (\ref{syst}) corresponds to a family of periodic solutions of the equation (\ref{lle}), which are parametrized by some $K\in \mathbb R$ and have the form
\[
\psi_{\mu,K}(x)=P(kx),
\]
where the $2\pi-$periodic function $P$ is given by
\[
P(y)=\psi^*+2\sqrt{\frac{-a_2 \mu-K^2}{b_2}}\left(C-\frac{4\omega^*}{1+2\psi_r^* \psi_i ^*}+i\right)\cos( y)+O(\mu),
\]
and $k=\omega^*+K+O(|\mu|^{\frac 1 2}).$

Finally, the truncated system (\ref{trunciomega})-(\ref{trunciomega1}) possesses a one-parameter family of homoclinic solutions, parametrized by some $K\in \mathbb R$. For $K=0$, this solution is given by
\[
A(x)=\sqrt{\frac{-a_2 \mu}{b_2}}\tanh\left( \sqrt{\frac{-a_2 \mu}{2}}|x|\right)e^{i\omega^* x} ,\quad B(x)=-\frac{a_2 \mu}{\sqrt{2 b_2}}\,{\rm{sech}}^2 \left( \sqrt{\frac{-a_2 \mu}{2}}|x|\right)e^{i\omega^* x}.
\]
The corresponding solution of (\ref{lle}) has the form
\begin{multline*}
\psi_{\mu}(x)=\psi^*+2(C+i)\sqrt{\left|\frac{a_2 }{b_2}\right|}\tanh\left(\sqrt{\frac{-a_2 \mu}{2}} |x| \right)\cos(\omega^* x)|\mu|^{\frac 1 2}\\-\frac{a_2}{\sqrt{2 b_2}}\,{\rm{sech}}^2 \left( \sqrt{\frac{-a_2 \mu}{2}}|x|\right)\sin(\omega^* x)\mu+O(|\mu|^{\frac 3 2}).
\end{multline*}

\section{Reversible $0^2(i\omega)$ bifurcation }\label{s3}
In this section we compute the normal form for the $0^2 (i\omega)$ bifurcation in the cases of normal ($\beta=1$) and anomalous ($\beta=-1$) dispersion, and then discuss the existence of stationary solutions of the equation (\ref{lle}). We show the existence of periodic, quasiperiodic orbits and homoclinic connections to periodic orbits.
\subsection{Case of normal dispersion ($\beta=1$)}\label{o2iomega1}

Consider the $0^2(i\omega)$ bifurcation, which occurs in the case $\beta=1$ for
\[\alpha^* >2,\quad F^{*2}=F^2_+(\alpha^*)=\frac{2\alpha^*-\gamma^*}{3} \left(1+\left(\frac{\alpha^*+\gamma^*}{3}\right)^2\right),\]
where $\gamma^*=\sqrt{\alpha^{*2}-3}$, \textit{i.e.}, the point $(\alpha^*, F^{*2})$ lies on the dotted line in Figure \ref{diagram} (case $\beta=1$). In this case, we have $\rho^*=\rho_+(\alpha^*)=(2\alpha^*-\gamma^*)/3$, in particular
\[3\rho^{*2}-4\alpha^* \rho^*+\alpha^{*2}+1=0,\]
so that 0 is an algebraically double eigenvalue of $L_+$. The other two eigenvalues $\pm i \omega^*$ of $L_+$ satisfy
\[
\omega^{*2}=2\alpha^*-4\rho^*=\frac 2 3 \left( 2\sqrt{\alpha^*-3}-\alpha^*\right)>0.
\]

As in the previous section we start by computing a convenient basis of $\mathbb R ^4$. According to \cite[Chapter 4, Lemma 3.3]{haragus2010local}, there exists a basis $\left\lbrace\zeta_0,\zeta_1,{\mathrm{Re}}\ \zeta,{\mathrm{Im}}\ \zeta\right\rbrace$ of $\rr^4$, such that
\[ L_+\zeta_0 =0, \quad L_+\zeta_1 =\zeta_0, \quad L_+\zeta=i\omega^* \zeta,\quad
S\zeta_0=\zeta_0, \quad S\zeta_1 =-\zeta _1,\quad S\zeta =\overline \zeta.
\]
A direct computation gives
\[
\zeta_0=\begin{pmatrix}
1\\0 \\D \\0
\end{pmatrix}, \quad
\zeta_1=\begin{pmatrix}
0\\ 1 \\ 0 \\ D
\end{pmatrix}, \quad \zeta=\begin{pmatrix}
1 \\i\omega^*\\ 0 \\ 0
\end{pmatrix},
\]
in which
\[D=\frac{3\psi_r^{*2}+\psi_i^{*2}-\alpha^*}{1-2\psi_r^* \psi_i^*}=-\frac{\omega^{*2}}{2}.
\]
In this basis we represent a vector $u\in \rr^4$ by $$ u=A\zeta_0 + B\zeta_1 +C \zeta +\overline C \overline{\zeta}, $$ with $A, B \in \rr$ and $C\in \cc$, by identifying $\rr^4$ with the space $\dis \rr^2 \times \widetilde {\rr^2}$, where
$\widetilde{\rr^2}=\left\lbrace  (C,\overline C), C\in \cc \right\rbrace.$

\vspace{0.3cm}
\noindent\textbf{Normal form.} The vector field in the reversible system (\ref{syst2}) is of class $\mathcal C^{\infty}$ on $ \rr^4 \times \rr$ and satisfies
\[
R_+(0,0)=0,\quad {\mathrm{d}}_U R_+(0,0)=0.
\]
We apply to the system (\ref{syst2}) the normal form theorem for reversible systems in the case of a $0^2(i\omega)$ bifurcation (see \cite[Chapter 3, Theorem 3.4]{haragus2010local} and \cite[Chapter 4, Lemma 3.5]{haragus2010local}). For any positive integer $N \geqslant 2$, there exist two neighborhoods $V_1$ and $V_2$ of 0 in $\rr^2 \times \widetilde{\rr^2}$ and $\rr$, respectively, and a polynomial $\Phi:V_1\times V_2 \longrightarrow \rr^4$ of degree $N$ with the properties:
\begin{enumerate}
\item $\Phi$ satisfies
\[\Phi(0,0,0,0,0)=0,\quad \partial_{(A,B,C,\overline C)}\Phi(0,0,0,0,0)=0, 
\]
and 
\[
\Phi(A,-B,\overline C,C,\mu)=S\Phi(A,B,C,\overline{C},\mu).
\]
\item For $(A,B,C,\overline C)\in V_1$, the change of variable 
\begin{equation}
U=A\zeta_0+B\zeta_1+ C\zeta +\overline{C}\overline{\zeta}+\Phi(A,B,C,\overline C, \mu)\label{chvar2}
\end{equation}
transforms the equation (\ref{syst2}) into the normal form
\begin{eqnarray}\label{normform}
\frac{\diff A}{\diff x}&=&B \nonumber \\
\frac{\diff B}{\diff x}&=&P(A,|C|^2,\mu)+\rho_B(A,B,C,\overline{C},\mu)\\
\frac{\diff C}{\diff x}&=&i\omega^* C+iCQ(A,|C|^2,\mu)+\rho_C(A,B,C,\overline{C},\mu)\nonumber,
\end{eqnarray}
where $P$ and $Q$ are real-valued polynomials of degree $N$ and $N-1$ in $(A,B,C,\overline C)$, respectively. Moreover the remainders $\rho_B$ and $\rho_C$ satisfy
\begin{eqnarray*}
\rho_B(A,-B,\overline{C},C,\mu)&=&\rho_B(A,B,C,\overline C,\mu),\\
\rho_C(A,-B,\overline{C},C,\mu)&=&-\overline{\rho_C}(A,B,C,\overline C,\mu),
\end{eqnarray*}
with the estimate
\[
\left|\rho_B(A,B,C,\overline C,\mu)\right|+\left|\rho_C(A,B,C,\overline C,\mu)\right|=o\left(\left(|A|+|B|+|C|\right)^N\right).
\]
\end{enumerate}
Consider the Taylor expansions of $P$ and $Q$ at order 2:
\begin{eqnarray}\label{expp}
P(A,|C|^2,\mu)=a_1\mu+b_1 A^2+c_1|C|^2+O(|\mu|^2+(|\mu|+|A|+|C|)^3),\\
Q(A,|C|^2,\mu)=a_2\mu+b_2 A+c_2 |C|^2+O((|\mu|+|A|+|C|^2)^2).\nonumber
\end{eqnarray}

The dynamics of the system (\ref{normform}) depends on the signs of the coefficients $a_1$, $b_1$ and $c_1$, provided that they do not vanish \cite[Chapter 4, Theorem 3.10]{haragus2010local}.

\begin{lem}\label{lemma2}
The coefficients $a_1$, $b_1$ and $c_1$ in the expansion of $P$ in (\ref{expp}) are given by

\begin{eqnarray*}
a_1= -\frac {9F^*(\alpha^*+\gamma^*)}{2(\alpha^{*3}-9\alpha^*+(\alpha^{*2}+6)\gamma^*)}
< 0,
\end{eqnarray*}
\begin{eqnarray*}
b_1=\frac {3F^*(2\alpha^{*5}-18\alpha^*+(2\alpha^{*4}+3\alpha^{*2}+9)\gamma^*)}{2(\alpha^{*2}+3+\alpha^* \gamma^*)(\alpha^{*3}-9\alpha^*+(\alpha^{*2}+6)\gamma^*)}
>0,
\end{eqnarray*}
\begin{eqnarray*}
c_1=\frac{9F^*(\alpha^*+\gamma^*)}{\alpha^{*3}-9\alpha^*+(\alpha^{*2}+6)\gamma^*}
>0.
\end{eqnarray*}
\end{lem}

\begin{proof}
%
We consider the vector $\zeta_1^*$ which satisfies 
\[
\left\langle\zeta_0,\zeta_1^*\right\rangle=0 \qquad \left\langle\zeta_1,\zeta_1^*\right\rangle=1, \qquad L_+^* \zeta_1^*=0,
\]
where $L_+^*$ is the adjoint operator of $L_+$.
A straightforward computation gives
\[
\zeta_1^*= \frac 1 D \begin{pmatrix}
0\\ 0 \\ 0 \\1
\end{pmatrix}.
\]
According to \cite[Chapter 4, Section 4.3.1]{haragus2010local}, the coefficients $a_1$, $b_1$ and $c_1$ are given by the formulas
\[
a_1= \left\langle R_+^{0,1}, \zeta_1^* \right\rangle,\quad b_1=\left\langle R_+^{2,0}(\zeta_0,\zeta_0),\zeta_1^* \right\rangle,\quad
c_1=\left\langle 2R_+^{2,0}(\zeta, \overline{\zeta}), \zeta_1^* \right\rangle,
\]
where
\[
R_+^{0,1}=\begin{pmatrix}
0 \\ -\psi_r^* \\ 0 \\ -\psi_i^*
\end{pmatrix},\quad R_+^{2,0}(\zeta_0,\zeta_0)=\begin{pmatrix}
0 \\ (D^2+3)\psi_r^*+2D\psi_i^* \\ 0 \\ 2D \psi_r^* +(3D^2+1)\psi_i^*
\end{pmatrix},\quad
R_+^{2,0}(\zeta,\overline \zeta)=\begin{pmatrix}
0 \\ 3\psi_r^* \\0 \\ \psi_i^*
\end{pmatrix}.
\]
We find 
\[
a_1=-\frac{\psi_i^*}{D}, \qquad b_1=\frac{2D\psi_r^*+(3D^2+1)\psi_i^*}{D}, \qquad c_1=\frac{2\psi_i^*}{D}.
\]
The explicit expressions of $a_1$, $b_1$ and $c_1$ in terms of $\alpha^*$ and $F^*$ and their signs are obtained with the symbolic package Maple.
\end{proof}
Here again, the reversibility of the system (\ref{syst2}) is a crucial argument in the proofs of the existence of solutions. Following \cite[Chapter 4, Theorem 3.10]{haragus2010local}, 
we state the following result:
\begin{theo}\label{thm}
Suppose that $\alpha^*>2$ and $F^{*2}=F_+^2(\alpha^*)$. In a neighborhood of $0$ in $\rr^4$ the following properties hold:
\begin{enumerate}
\item \label{thmii} For any $\mu>0$ sufficiently small, the system (\ref{syst2}) possesses two equilibria, a center and a saddle-center, together with two one-parameter families of periodic orbits, called periodic orbits of the first kind, parametrized by their size $r$, for $r<r^*(\mu)=O( \mu^{\frac 1 2})$, which tend to the two equilibria when $r$ tends to 0. For any periodic orbit in the family which tends to the saddle-equilibrium, with size $r$ not too small, $r> r_*(\mu)$, there is a pair of reversible homoclinic orbits connecting the periodic orbit to itself. 
In addition there are also periodic orbits called periodic orbits of the second kind and quasiperiodic orbits.
\item For any $\mu<0$ sufficiently small, the system (\ref{syst2}) has no bounded solution. 
\end{enumerate}
\end{theo}

\noindent\textbf{Solutions. }As before we can compute the leading order terms in the Taylor expansion of the corresponding solutions of the equation (\ref{lle}). 
Consider the truncated system
\begin{eqnarray}\label{trunciomega2}
\frac{\diff A}{\diff x}&=&B \nonumber \\
\frac{\diff B}{\diff x}&=&a_1\mu+b_1 A^2+c_1|C|^2\\
\frac{\diff C}{\diff x}&=&i\omega^* C\nonumber,
\end{eqnarray}
and assume that $\mu >0$. Notice that $K=|C|^2$ is a first integral of this system.

The system (\ref{trunciomega2}) has two equilibria $(\pm A_0,0,0,0)$, with $A_0=\sqrt{-a_1 \mu/b_1}$. The equilibrium $(A_0,0,0)$ is a saddle and $(-A_0,0,0)$ is a center.
These equilibria correspond to constant solutions of the equation (\ref{lle}) with parameters $\alpha^* +\mu$ and $F^*$, and read
\[
\psi(x)=\psi^*\pm(1+iD)\sqrt{\frac{-a_1 }{b_1}}\mu^{\frac 1 2}+O(\mu).
\]

In addition, the system (\ref{trunciomega2}) possesses two families of periodic orbits $(\pm A_K, 0, C(x), \overline C (x))$, where
\[
A_K=\sqrt{\frac{-a_1 \mu -c_1 K}{b_1}},\quad C(x)=\sqrt K e^{i \omega^* x},
\]
in which $K$ is such that $a_1 \mu +c_1 K<0$. These periodic orbits are called periodic orbits of the first kind. The corresponding solutions of the equation (\ref{lle}) are periodic solutions of the form
\[
\psi_{\mu,\pm,K}(x)=P_1(kx),
\]
where the $2\pi-$periodic function $P_1$ has the expansion
\[
P_1(y)=\psi^*\pm (1+Di)\sqrt{\frac{-a_1 \mu -c_1 K}{b_1}}+2\cos(y)+O(\mu)
\]
and $k=\omega^*+O(\mu^{\frac 1 2})$.

Next, the system (\ref{trunciomega2}) admits a third family of periodic orbits, which are obtained with $C=0$ and are called periodic orbits of the second kind. These solutions have the form $(A(x), B(x), 0,0)$, where $B=\diff A /\diff x$ and $A$ satisfies the differential equation
\[
\frac{\diff^2 A}{\diff x^2}=a_1 \mu + b_1 A^2.
\]
Remark that the equilibrium $-A_0$ is a center for this equation, and is surrounded by a one-parameter family of periodic orbits. The solutions which are close to the center have the expansion
\[
A(x)=-\sqrt{\frac{-a_1\mu}{b_1}}+\varepsilon \sqrt{\mu}\cos(\sqrt 2 (-a_1b_1\mu)^{\frac 1 4} x)+O(\varepsilon^2 ),
\]
where $\varepsilon$ is a real and small parameter.
The corresponding solutions of the equation (\ref{lle}) have the form
\[
\psi_{\mu,\varepsilon}(x)=P_2(kx),
\]
where the $2\pi-$periodic function $P_2$ is given by
\[
P_2(y)=\psi^*+(1+Di)\left(-\sqrt{\frac{-a_1}{b_1}}+\varepsilon \cos\left(y \right) \right)\mu^{\frac 1 2}+O(\mu+\varepsilon(\varepsilon+\mu))
\]
and $k=\sqrt 2 (-a_1b_1\mu)^{\frac 1 4}+O(\mu^{\frac 1 2})$.

Finally, the saddle $(A_0,0,0,0)$ is surrounded by the one-parameter family of periodic orbits $(A_K,0,C(x), \overline C (x))$ above, and to each of this periodic orbits there exists a one-parameter family of homoclinic orbit connecting this periodic orbit to itself. This family of homoclinic orbits reads
\[
A(x)=A_K\left(1-3\,{\rm{sech}}^2\left(\delta_K x\right)\right), \qquad B(x)=\frac{\diff A}{\diff x},\qquad C(x)=\sqrt{K}e^{i (\omega^* x+ \varphi)},
\]
where $\delta_K=\sqrt{b_1 A_K/2}$ and $\varphi \in \mathbb R / 2\pi \mathbb Z$.

The question of the persistence of these solutions as solutions of the full system (\ref{normform}) is a delicate issue, which has been studied in \cite[Chapter 7, Section 7.3]{lombardi2000oscillatory}. It turns out that the reversible orbits (\textit{i.e.}, the orbits obtained with $\varphi=0, \pi$) persist, provided that $K$ is not too small.
In particular, the homoclinic solution $(A(x), B(x),0,0)$ of (\ref{trunciomega2}), with \[
A(x)=A_0 \left(1-3\,{\rm{sech}}^2\left(\delta_0 x\right)\right), \qquad B(x)=\frac{\diff A}{\diff x},
\]
does not persist as a homoclinic solution of the system (\ref{normform}) (see \cite[Chapter 7, Section 7.4]{lombardi2000oscillatory}).

The corresponding solutions of the equation (\ref{lle}) have the form
\[
\psi_{\mu, K, \pm}(x)=\psi^*+(1+iD)\sqrt{\frac{-a_1 \mu -c_1 K}{b_1}}\left(1-3\,{\rm{sech}}^2 \left(\delta_K x\right)\right)\pm2\sqrt K \cos(kx)+O(\mu),
\]
where $k=\omega^*+O(\mu^{\frac 1 2})$.

\subsection{Case of anomalous dispersion ($\beta=-1$)}\label{o2iomega2}

We keep the notations of Section \ref{o2iomega1}. In this case we have to consider two cases for the parameters $\alpha^*$ and $F^*$:

\begin{enumerate}
\item $ \alpha^* \in \left(\sqrt 3 , 2\right)$ and $\dis F^{*2}=F^2_+(\alpha^*)$,\label{case3}
\item$\dis \alpha^* > \sqrt 3$ and $\dis F^{*2}=F_{-}^2(\alpha^*)$, \label{case4}
\end{enumerate}
that is, $(\alpha^*, F^{*2})$ lies on the dotted line on Figure \ref{diagram} (case $\beta=-1$). The case $\alpha^*=\sqrt 3$ is excluded here, because at this point one of the coefficient of the normal form computed hereafter vanishes.

In both cases, we have \[3\rho^{*2}-4\alpha^* \rho^*+\alpha^{*2}+1=0,\]
so that 0 is an algebraically double eigenvalue of $L_-$. The other two eigenvalues $\pm i \omega^*$ of $L_-$ satisfy
\[
\omega^{*2}=4\rho^*-2\alpha^*.
\]
\noindent\textbf{Case \ref{case3}.} In this case $\rho^*=\rho_+(\alpha^*)$ and

\[
\omega^{*2}=\frac 2 3 (\alpha^*-2\gamma^*)=\frac{2}{3}(\alpha^*-2\sqrt{\alpha^{*2}-3}).
\]
The system (\ref{syst2}) has the normal form (\ref{normform}), and the coefficients $a_1$, $b_1$ and $c_1$ can be computed as in the proof of Lemma \ref{lemma2}. We obtain the following result:

\begin{lem}The coefficients $a_1$, $b_1$ and $c_1$ in the expansion of $P$ in (\ref{expp}) are given by
\begin{eqnarray*}
a_1= \frac {9F^*(\alpha^*+\gamma^*)}{2(\alpha^{*3}-9\alpha^*+(\alpha^{*2}+6)\gamma^*)}  < 0,
\end{eqnarray*}

\begin{eqnarray*}
b_1= -\frac {3F^*(2\alpha^{*5}-18\alpha^*+(2\alpha^{*4}+3\alpha^{*2}+9)\gamma^*)}{2(\alpha^{*2}+3+\alpha^* \gamma^*)(\alpha^{*3}-9\alpha^*+(\alpha^{*2}+6)\gamma^*)}> 0,
\end{eqnarray*}

\begin{eqnarray*}
c_1=-\frac{9F^*(\alpha^*+\gamma^*)}{\alpha^{*3}-9\alpha^*+(\alpha^{*2}+6)\gamma^*}
>0.
\end{eqnarray*}
\end{lem}
Then following \cite[Chapter 4, Theorem 3.10]{haragus2010local}, we obtain the same types of solutions as in Theorem \ref{thm}, and the leading order terms of the corresponding solutions of the equation (\ref{lle}) are the same as the one found in Section \ref{o2iomega1}.

\vspace{0.3cm}
\noindent\textbf{Case \ref{case4}.} In this case $\rho^*=\rho_-(\alpha^*)= (2\alpha^*+\gamma^*)/3$ and
\[
\omega^{*2}=\frac 2 3 (\alpha^*+2\sqrt{\alpha^{*2}-3}).
\]
By arguing as before we obtain the following lemma:

\begin{lem}The coefficients $a_1$, $b_1$ and $c_1$ in the expansion of $P$ in (\ref{expp}) are given by
\begin{eqnarray*}
a_1=\frac {9F^*(\alpha^*-\gamma^*)}{2(\alpha^{*3}-9\alpha^*-(\alpha^{*2}+6)\gamma^*)} < 0,
\end{eqnarray*}

\begin{eqnarray*}
b_1=-\frac {3F^*(2\alpha^{*5}-18\alpha^*-(2\alpha^{*4}+3\alpha^{*2}+9)\gamma^*)}{2(\alpha^{*2}+3-\alpha^* \gamma^*)(\alpha^{*3}-9\alpha^*-(\alpha^{*2}+6)\gamma^*)}<0,
\end{eqnarray*}
 
\begin{eqnarray*}
c_1=-\frac{9F^*(\alpha^*-\gamma^*)}{\alpha^{*3}-9\alpha^*-(\alpha^{*2}+6)\gamma^*}
>0.
\end{eqnarray*}
\end{lem}

According to \cite[Chapter 4, Theorem 3.10]{haragus2010local}, we obtain the following theorem:

\begin{theo}
\begin{enumerate}
Suppose that $\alpha^*>\sqrt 3$ and $F^{*2}=F_-^2(\alpha^*)$. In a neighborhood of $0$ in $\rr^4$ the following properties hold:
\item For any $\mu>0$ sufficiently small, there exist two families of periodic orbits of the first kind, parametrized by their size $r$, with $\dis r>r^*(\mu)=O(\mu^{\frac 1 2})$. To any periodic orbit in one of these families, there is a pair of reversible homoclinic orbits connecting the periodic orbit to itself. In addition there are also periodic orbits of the second kind and quasiperiodic orbits.
\item For any $\mu<0$ sufficiently small, the system (\ref{syst2}) possesses the same types of solutions as in Theorem \ref{thm} \ref{thmii}.
\end{enumerate}
\end{theo}

\noindent\textbf{Solutions. }We start with the truncated system (\ref{trunciomega2}) and keep the notations of Section \ref{o2iomega1}. First, suppose that $\mu<0$. 
The system (\ref{trunciomega2}) admits two equilibria $(\pm A_0, 0,0,0)$. Notice that $( A_0, 0,0,0)$ is a center, while $(-A_0, 0,0,0)$ is a saddle. The formulas for the corresponding constant solutions of the equation (\ref{lle}) are the same as in Section \ref{o2iomega1}. 

Next, the formulas for the solutions of (\ref{lle}) corresponding to the periodic orbits of the first kind of the system (\ref{syst2}) are the same as in Section \ref{o2iomega1}.
The center $( A_0, 0,0,0)$ is surrounded by a one-parameter family of periodic orbits of the second kind. As previously, an expansion of the solutions which are close to the center can be computed. The corresponding solutions of the equation (\ref{lle}) have the form
\[
\psi_{\mu,\varepsilon}(x)=P_2(kx),
\]
where 
\[
P_2(y)=\psi^*+(1+Di)\left(\sqrt{\frac{a_1\mu}{b_1}}+\varepsilon \cos\left(y \right) \right)|\mu|^{\frac 1 2}+O(\mu+\varepsilon(\varepsilon+\mu))
\]
with $k=\sqrt 2 (-a_1b_1\mu)^{\frac 1 4}+O(|\mu|^{\frac 1 2})$ and $\varepsilon$ is a real and small parameter.

Finally, the solutions of (\ref{lle}) which correspond to the pair of reversible homoclinic connections to periodic orbits have the form

\[
\psi_{\mu, K, \pm}(x)=\psi^*-(1+iD)\sqrt{\frac{-a_1 \mu -c_1 K}{b_1}}\left(1-3\,{\rm{sech}}^2\left(\delta_K x\right)\right)\pm2\sqrt K \cos(kx)+O(\mu),
\]
where $k=\omega^*+O(|\mu|^{\frac 1 2})$ and $\delta_K=\sqrt{-b_1 A_K/2}$.

Now suppose that $\mu >0$. The formulas for the solutions of (\ref{lle}) corresponding to the periodic orbits of the first kind of the system (\ref{syst2}) are the same as in Section \ref{o2iomega1}. 
The solutions of (\ref{lle}) which correspond to the pair of homoclinic connections to periodic orbits have the form
\[
\psi_{\mu, K,\pm}(x)=\psi^*+(1+iD)A_K^{1,2}\left(1-3\,{\rm{sech}}^2\left(\delta_K x\right)\right)\pm 2\sqrt K \cos(kx)+O(\mu),
\]
where $k=\omega^*+O(\mu^{\frac 1 2})$, $A_K^1=A_K$, $A_K^2=-A_K$ and $\delta_K=\sqrt{|b_1 A_K|/2}$.
\section{Reversible $0^2$ bifurcation}\label{s4}
In this section we analyse the $0^2$ bifurcation in the cases of normal and anomalous dispersion. This analysis relies upon a center manifold reduction. We prove the existence of periodic and homoclinic solutions.

\subsection{Case of normal dispersion ($\beta=1$)}\label{o21}

Consider the $0^2$ bifurcation, which occurs in the case $\beta=1$ in the two cases
\begin{enumerate}
\item \label{case1}$ \alpha ^* \in \left(\sqrt 3 , 2\right)$ and $\dis F^{*2}=F_+ ^2 (\alpha^*)$  
\item \label{case2}$ \alpha^* >\sqrt 3$ and $\dis F^{*2}=F_- ^2 (\alpha^*)$, 
\end{enumerate}
\textit{i.e.} the point $(\alpha^*,F^{*2})$ lies on the continuous line in Figure \ref{diagram} (case $\beta=1$). The case $\alpha^*=\sqrt 3$ is excluded here, because at this point one of the coefficient of the normal form computed hereafter vanishes.
In both cases 0 is an non-semisimple eigenvalue of $L_+$ with algebraic multiplicity $2$ and $L_+$ has no other purely imaginary eigenvalue. Indeed, the two nonzero eigenvalues $\pm \omega^*$ of $L_+$ satisfy, in Case \ref{case1},
\[
\omega^{*2}=\frac{2}{3}(\alpha^{*2}-2\sqrt{\alpha^{*2}-3})>0,
\]
and in Case \ref{case2},
\[
\omega^{*2}=\frac{2}{3}(\alpha^{*2}+2\sqrt{\alpha^{*2}-3})>0.
\]
The analysis of this bifurcation relies upon a center manifold reduction. 

\vspace{0.3cm}
\noindent\textbf{Center manifold reduction. }Let $\zeta_0$ and $\zeta_1$ be two generalized eigenvectors of $L_+$, satisfying
\[
L_+\zeta_0=0,\quad L_+ \zeta_1=\zeta_0, \quad S\zeta_0=\zeta_0,\quad S\zeta_1=-\zeta_1.
\]
A straightforward computation gives
\[
\zeta_0=\begin{pmatrix}
1\\0 \\D \\0
\end{pmatrix}, \quad
\zeta_1=\begin{pmatrix}
0\\ 1 \\ 0 \\ D
\end{pmatrix},\quad D=\frac{3\psi_r^{*2}+\psi_i^{*2}-\alpha^*}{1-2\psi_r^* \psi_i^*}=-\frac{\omega^{*2}}{2}.
\]

Consider the spectral decomposition $\mathbb R^4=X_0 \oplus X_1$, with \[X_0={\mathrm{span}}(\zeta_0, \zeta_1), \quad X_1=({\mathrm{id}}-P_0)(\mathbb R ^4),\] where $P_0$ is the unique projection onto $X_0$ which commutes with $L_+$.
Recall that $R_+$ is of class  $\mathcal C ^{\infty}$, and $R_+(0,0)=0$, ${\mathrm{d}}_U R_+(0,0)=0$. According to the center manifold theorem (see for example \cite[Chapter 2, Theorem 3.3]{haragus2010local}), there exist a map $\Psi\in \mathcal C^k (X_0\times \rr, X_1)$, for any integer $k$, satisfying
\[
\Psi(0,0)=0,\quad \diff_U \Psi(0,0)=0,
\]
and a neighborhood $O_1 \times O_2$ of $(0,0)$ in $\rr^4 \times \rr$ such that for $\mu \in O_2$, the manifold
\[
\mathcal M (\mu)=\left\lbrace U_0+\Psi(U_0, \mu),\  U_0\in X_0 \right\rbrace
\]
is locally invariant, and contains the set of bounded solutions of the system (\ref{syst2}) staying in $O_1$ for all $x\in \rr$.
Consequently, any small bounded solution $U$ of (\ref{syst2}) has the form

\[U=U_0+\Psi(U_0,\mu),\quad U_0=P_0 U,\quad \Psi(U_0,\mu)=(\id-P_0)U,
\]
with \[U_0(x)=A(x)\zeta_0+B(x) \zeta_1,\quad\Psi(U_0,\mu)=O(|\mu|+\left \|U_0 \right \|(|\mu|+\left\|U_0\right \|)).
\]
where $A,\ B$ are real-valued functions. 
Moreover, according to \cite[Chapter 2, Corollary 2.12]{haragus2010local}, the function $U_0$ satisfies the reduced system 
\begin{equation}
\frac{\diff U_0}{\diff x}=L_{+,0} U_0 +P_0 R_+(U_0+\Psi(U_0,\mu),\mu),\label{systreduit}
\end{equation}   
where $L_{+,0}$ is the restriction of $L_+$ to $X_0$. Our goal is now to compute the system (\ref{systreduit}).

In the basis $\left(\zeta_0, \ \zeta_1\right )$ of $X_0$ the operator $L_{+,0}$ is represented by the $2\times 2$ matrix 
\[
L_{+,0}=\begin{pmatrix}
0 & 1 \\ 0 & 0
\end{pmatrix}.
\]
Next, in order to compute $P_0$, we consider the vectors $\zeta_0^*, \ \zeta_1^*$ satisfying
\[
\left\langle \zeta_1^*,\zeta_0\right\rangle =0,\quad \left\langle \zeta_1^*,\zeta_1\right\rangle =1,\quad L_+^*\zeta_1^*=0,\quad L_+^*\zeta_0^*=\zeta_1^*,
\]
where $L_+^*$ is the adjoint operator of $L_+$. A direct computation provides
\[
\zeta_0^*=\frac{1}{D}\begin{pmatrix}0\\0\\1\\0\end{pmatrix},\qquad \zeta_1^*=\frac{1}{D}\begin{pmatrix}0\\ 0 \\ 0 \\1\end{pmatrix}.
\]
Then the projector $P_0$ is given, for any $V\in \rr^4$, by the formula
\[
P_0 V=\langle V,\zeta_0^*\rangle \zeta_0+\langle V,\zeta_1^*\rangle \zeta_1.
\] 
Consequently,
\[
P_0 R_+(U_0+\Psi(U_0,\mu),\mu)=\left\langle R_+(U_0+\Psi(U_0,\mu),\mu),\zeta_0^* \right\rangle \zeta_0+\left\langle R_+(U_0+\Psi(U_0,\mu),\mu),\zeta_1^* \right\rangle \zeta_1.
\]
Since the second and the fourth component of $\zeta_0^*$ are zero, we have
\[
\left\langle R_+(U_0+\Psi(U_0,\mu),\mu),\zeta_0^* \right\rangle=0,
\]
so that
\[
P_0 R_+(U_0+\Psi(U_0,\mu),\mu)=\left\langle R_+(U_0+\Psi(U_0,\mu),\mu),\zeta_1^* \right\rangle \zeta_1.
\]
Furthermore, since $\dis \Psi(U_0,\mu)=O(\left \|U_0 \right \|(|\mu|+\left\|U_0\right \|))$ and $R_+(U,\mu)=O(|\mu|+\| U\|^2)$, we obtain
\[P_0 R_+(U_0+\Psi(U_0,\mu),\mu)=\left\langle R_+(U_0,\mu),\zeta_1^* \right\rangle \zeta_1+O(\left \|U_0 \right \|(|\mu|+\left\|U_0\right \|^2)).
\]
Using the expression of $R_+$, we have
\begin{eqnarray*}
R_+(U_0,\mu)&=&R_+(A\zeta_0+B\zeta_1,\mu)\\&=&\begin{pmatrix}
0 \\ -\mu \psi_r^* +((D^2+3)\psi_r^*+2D\psi_i^*)A^2 \\ 0 \\ -\mu\psi_i^* +(2D\psi_r^*+(3D^2+1)\psi_i^*)A^2
\end{pmatrix}+O(\left \|U_0 \right \|(|\mu|+\left\|U_0\right \|^2)).
\end{eqnarray*}
Finally, the reduced system (\ref{systreduit}) reads
\begin{eqnarray}
\frac{\diff A}{\diff x}&=&B\label{redsyst1}\\
\frac{\diff B}{\diff x}&=&a\mu+bA^2+O\left(\left(|A|+|B| \right)\left(|\mu|+\left(|A|^2+|B|^2\right)\right)\right),\label{redsyst2}
\end{eqnarray}
with
\[a=-\frac{\psi_i ^*}{D},\qquad b=\frac{2D\psi_r^*+(3D^2+1)\psi_i^*}{D}.\]

Notice that the system the system (\ref{redsyst1})-(\ref{redsyst2}) is already in normal form at order 2. According to \cite[Chapter 4, Section 4.1.1, Theorem 1.8]{haragus2010local}, its dynamics depends on the signs of the coefficients $a$ and $b$, provided that they do not vanish.

\vspace{0.3cm}
\noindent\textbf{Case \ref{case1}.} We suppose that $\alpha^* \in (\sqrt 3 , 2)$ and $F^{*2}=F^2_+(\alpha^*)$, so that $\rho^*=\rho_+(\alpha^*)$. A direct computation, using the symbolic package Maple, gives
\begin{eqnarray*}
a=-\frac {9F^*(\alpha^*+\gamma^*)}{2(\alpha^{*3}-9\alpha^*+(\alpha^{*2}+6)\gamma^*)}
>0,
\end{eqnarray*}

\begin{eqnarray*}
b=\frac {3F^*(2\alpha^{*5}-18\alpha^*+(2\alpha^{*4}+3\alpha^{*2}+9)\gamma^*)}{2(\alpha^{*2}+3+\alpha^* \gamma^*)(\alpha^{*3}-9\alpha^*+(\alpha^{*2}+6)\gamma^*)}
<0.
\end{eqnarray*}
We therefore have the following result (see \cite{haragus2010local}, Chapter 4, Theorem 1.8), which is valid in a neighborhood of the origin in $\rr^2$, for small $\mu$:

\begin{theo}\label{thm02}
\begin{enumerate}
\item For $\mu>0$ sufficiently small, the system (\ref{syst2}) possesses two equilibria of order $O(\mu^{\frac 1 2})$, a saddle and a center. The center is surrounded by a one-parameter family of periodic orbits, which tend to a homoclinic orbit connecting the saddle equilibrium to itself, as the period tends to $\infty$.
\item For $\mu<0$ sufficiently small, the system (\ref{syst2}) has no bounded solution. 
\end{enumerate}
\end{theo}

\noindent\textbf{Solutions. }As before we can compute the leading order terms in the Taylor expansions of the corresponding solutions of the equation (\ref{lle}). Consider the truncated system
\begin{eqnarray}
\frac{\diff A}{\diff x}&=&B\label{troncO2} \\
\frac{\diff B}{\diff x}&=&a\mu+bA^2.\label{troncO21}
\end{eqnarray}

For $\mu>0$, the system (\ref{troncO2})-(\ref{troncO21}) possesses two equilibria $(\pm A_0,0)$, where
$A_0=\sqrt{-a\mu/b}$. 
The equilibrium $(A_0,0)$ is a center and $(-A_0,0)$ is a saddle. These equilibria correspond to constant solutions of the equation (\ref{lle}) with corresponding parameters $\alpha^*+\mu$ and $F^*$, and read
\[
\psi_\mu=\psi^*\pm \sqrt{\frac{-a}{b}}(1+iD)\mu+O(\mu).
\]
Next, the small periodic solutions $(A,B)$ of (\ref{troncO2})-(\ref{troncO21}) which are close to the center have the form
\[
A(x)=A_0+\varepsilon \sqrt{\mu}\cos\left(\sqrt 2 (-ab\mu)^{\frac 1 4}x \right)+O(\varepsilon ^2 ),\quad B(x)=\frac{\diff A}{\diff x},
\]
where $\varepsilon$ is a real and small parameter. The corresponding solutions of (\ref{lle}) have the form
\[
\psi_{\mu,\varepsilon}(x)=P(kx),\]
where $P$ is the $2\pi-$ periodic function defined by
\[
P(y)=\psi^*+(1+Di)\left(\sqrt{\frac{-a }{b}}+\varepsilon \cos\left(y \right) \right)\mu^{\frac 1 2}+O(\mu+\varepsilon(\varepsilon+\mu)),
\]
and $k=\sqrt 2 (-ab\mu)^{\frac 1 4}+O(\mu^{\frac 1 2})$.

Finally, the homoclinic solution $(A,B)$ of the system (\ref{troncO2})-(\ref{troncO21}) is given by
\[
A(x)=-A_0 \left(1-3 \, {\rm{sech}}^2 (\delta x)\right), \quad B(x)=\frac{\diff A}{\diff x},
\]
where \[\delta = \sqrt{\frac{-b A_0}{2}}.\]
The corresponding solution of (\ref{lle}) reads
\[
\psi_{\mu}(x)=\psi^*-\sqrt{\frac{-a}{b}}(1+Di)\left(1-3\, {\rm{sech}}^2(\delta x) \right)\mu^{\frac 1 2}+O(\mu).
\]

\vspace{0.3cm}
\noindent\textbf{Case \ref{case2}.} We suppose now that $\alpha^* > \sqrt 3$ and $F^{*2}=F_-^2(\alpha^*)$, so that $\rho^*=\rho_-(\alpha^*)$. In this case the coefficients $a$ and $b$ are given by
\begin{eqnarray*}
a=-\frac {9F^*(\alpha^*-\gamma^*)}{2(\alpha^{*3}-9\alpha^*-(\alpha^{*2}+6)\gamma^*)}
>0
\end{eqnarray*}
and 
\begin{eqnarray*}
b=\frac {3F^*(2\alpha^{*5}-18\alpha^*-(2\alpha^{*4}+3\alpha^{*2}+9)\gamma^*)}{2(\alpha^{*2}+3-\alpha^* \gamma^*)(\alpha^{*3}-9\alpha^*-(\alpha^{*2}+6)\gamma^*)}
>0.
\end{eqnarray*}

Then we have the following theorem:
\begin{theo}
\begin{enumerate} 
\item For $\mu>0$ sufficiently small, the system (\ref{syst2}) has no bounded solution.
\item For $\mu<0$ sufficiently small, the system (\ref{syst2}) possesses two equilibria of order $O(|\mu|^{\frac 1 2})$, a saddle and a center. The center is surrounded by a one-parameter family of periodic orbits, which tend to a homoclinic orbit connecting the saddle equilibrium to itself, as the period tends to~$\infty$.
\end{enumerate}
\end{theo}

\noindent\textbf{Solutions. }We consider the truncated system (\ref{troncO2})-(\ref{troncO21}). For $\mu <0$, this system possesses two equilibria $(\pm A_0,0)$, where $A_0=\sqrt{-a\mu/b}$. In this case $(A_0,0)$ is a saddle and $(-A_0,0)$ is a center. These equilibria correspond to constant solutions of the equation (\ref{lle}) with corresponding parameters $\alpha^*+\mu$ and $F^*$, and read
\[
\psi=\psi^*\pm \sqrt{\frac{a}{b}}(1+Di)|\mu|^{\frac 1 2}+O(\mu).
\] 
Moreover, as in Case \ref{case1}, the equation (\ref{lle}) has a family of periodic solutions. The solutions which are close to $\psi^*$ have the form 
\[
\psi_{\mu,\varepsilon}(x)=P(kx), \]
where $k=\sqrt 2 (-ab\mu)^{\frac 1 4}+O(|\mu|^{\frac 1 2})$, and the $2\pi-$periodic function $P$ is given by
\[
P(y)=\psi^*+(1+Di)\left(-\sqrt{\frac{a }{b}}+\varepsilon \cos\left(y \right) \right)|\mu|^{\frac 1 2}+O(\mu+\varepsilon(\varepsilon+\mu)),
\]
in which $\varepsilon$ is a real and small parameter. The equation (\ref{lle}) also have a homoclinic solution, which reads
\[
\psi_{\mu}(x)=\psi^*+(1+Di)\sqrt{\frac{a}{b}}\left(1-3\ {\rm{sech}}^2(\delta x) \right)|\mu|^{\frac 1 2}+O(\mu),
\]
where $\delta=\sqrt{bA_0 /2}$.

\subsection{Case of anomalous dispersion ($\beta=-1$)}
We keep the notations of Section \ref{o21}. Consider the $0^2$ bifurcation, which occurs for $\dis \alpha^*>2$ and $\dis F^{*2}=F_+ ^2(\alpha^*)$, \textit{i.e.}, the point $(\alpha^*,F^{*2})$ lies on the continuous line on Figure \ref{diagram} (case $\beta=-1$). In this case we have $\rho^*=\rho_+(\alpha^*)$. As in the previous section, the study of this bifurcation also requires a center manifold reduction. The reduced system has the same form as the one found in \ref{o21}, and the coefficients at leading order are given by
\begin{eqnarray*}
a=\frac {9F^*(\alpha^*+\gamma^*)}{2(\alpha^{*3}-9\alpha^*+(\alpha^{*2}+6)\gamma^*)}>0
\end{eqnarray*}
and 
\begin{eqnarray*}
b= -\frac {3F^*(2\alpha^{*5}-18\alpha^*+(2\alpha^{*4}+3\alpha^{*2}+9)\gamma^*)}{2(\alpha^{*2}+3+\alpha^* \gamma^*)(\alpha^{*3}-9\alpha^*+(\alpha^{*2}+6)\gamma^*)}<0.
\end{eqnarray*}
In this case we obtain the same types of solutions as in Theorem \ref{thm02} and the same formulas for the corresponding solutions of the equation (\ref{lle}) as in Section \ref{o21}, Case \ref{case1}.

\section{Discussion}

\noindent\textbf{Time periodic solutions. }Our analysis confirms the existence of several types of steady solutions of the equation (\ref{lle}), which have been experimentally and numerically observed (see \cite{godey2014stability} and the references therein). For example, several families of periodic solutions are found for different values of the parameters $\alpha$ and $F$, and dark solitons, with a shape corresponding to homoclinic solutions found in the $0^2$ bifurcation, are also observed.

Particular solutions which are periodic in time, such as breathers, have been experimentally and numerically observed \cite{godey2014stability}. There are no time periodic solutions bifurcating from constant solutions. Indeed, one can perform the same analysis and find that the only small solutions are time-independant. A possible way
of finding time periodic solutions could be through a study of secondary bifurcations.

\vspace{0.3cm}
\noindent\textbf{Stability.}
The spectral stability of constant solutions of the equation (\ref{lle}) is well-known \cite{godey2014stability}. Straightforward computations prove that when the equation (\ref{lle}) has one constant solution, this one is spectrally stable, while when it possesses three constant solutions $\psi_1$, $\psi_2$ and $\psi_3$, with $\rho_1<\rho_2<\rho_3$, the constant solutions $\psi_1$, $\psi_3$ are spectrally stable, and $\psi_2$ is spectrally unstable.

Concerning non-constant solutions, the only existing results concern a family of periodic waves, obtained in a neighborhood of $(\rho,\alpha) =(1,41/30)$ in the case $\beta=-1$. More precisely, these solutions are nonlinearly stable if $\alpha <41/30$ and nonlinearly unstable if $\alpha >41/30$ (see \cite{miyaji2010bifurcation}). Moreover, it is proved in \cite{miyaji2011stability}, using a Strichartz estimate, that this family is orbitaly stable with respect to $L^2$ perturbations for $\alpha<41/30$.

The question of the stability for other types of solutions, in particular for the stationary solutions found in our analysis, is widely open.

\end{document}